\documentclass[11pt,fleqn]{amsart}
\usepackage{amssymb}
\usepackage{mathrsfs}
\usepackage[all]{xy}
\usepackage{color}
\usepackage{soul}
\usepackage{hyperref}
\usepackage{dsfont}

\newcommand\N{{\mathbb N}}
\newcommand\R{{\mathbb R}}

\newcommand\C{{\mathbb C}}

\newcommand\Z{{\mathbb Z}}

\newcommand\Sp{{\mathbb S}}

\def\AA{{\mathcal A}}
\def\BB{{\mathcal B}}
\def\CC{{\mathcal C}}
\def\DD{{\mathcal D}}
\def\EE{{\mathcal E}}

\def\LL{{\mathcal L}}

\def\NN{{\mathcal N}}

\def\RR{{\mathcal R}}

\def\BBB{{\mathscr B}}
\def\CCC{{\mathscr C}}
\def\DDD{{\mathscr D}}

\newtheorem{theo}{Theorem}
\newtheorem{prop}[theo]{Proposition}
\newtheorem{lem}[theo]{Lemma}
\newtheorem{cor}[theo]{Corollary}
\newtheorem{rem}[theo]{Remark}

\newtheorem{defin}[theo]{Definition}

\newcommand{\beqn}{\begin{equation}}
\newcommand{\eeqn}{\end{equation}}
\newcommand{\bear}{\begin{eqnarray}}
\newcommand{\eear}{\end{eqnarray}}
\newcommand{\bean}{\begin{eqnarray*}}
\newcommand{\eean}{\end{eqnarray*}}

\title[Exponential convergence for the Boltzmann equation without cut-off]{Exponential convergence to equilibrium for the homogeneous Boltzmann equation for hard potentials without cut-off}

\begin{document}

\author{{\sc Isabelle Tristani}}
\address{CEREMADE, Universit\'e Paris IX-Dauphine,
Place du Mar\'echal de Lattre de Tassigny, 75775 Paris Cedex 16, France. 
E-mail: {\tt tristani@ceremade.dauphine.fr }}

\date\today

\maketitle

\begin{abstract} 
This paper deals with the long time behavior of solutions to the spatially homogeneous Boltzmann equation. The interactions considered are the so-called (non cut-off and non mollified) hard potentials, we thus only deal with a moderate angular singularity. We prove an exponential in time convergence towards the equilibrium, improving results of Villani from \cite{Vill1} where a polynomial decay to equilibrium is proven. The basis of the proof is the study of the linearized equation for which we prove a new spectral gap estimate in a $L^1$ space with a polynomial weight by taking advantage of the theory of enlargement of the functional space for the semigroup decay developed by Gualdani et al. in \cite{GMM}. We then get our final result by combining this new spectral gap estimate with bilinear estimates on the collisional operator that we establish.
\end{abstract}

\vspace{1cm}

\vspace{0.3cm}

\vspace{0.5cm}
\tableofcontents

\newpage


\section{Introduction} 
\label{sec:intro}
\setcounter{equation}{0}
\setcounter{theo}{0}


\subsection{The model}
In the present paper, we investigate the asymptotic behavior of solutions to the spatially homogeneous
Boltzmann equation without angular cut-off, that is, for long-range interactions. Previous works have shown that these solutions converge towards the Maxwellian equilibrium with a polynomial rate when time goes to infinity. Here, we are interested in improving the rate of convergence and we show an exponential decay to equilibrium. 

We consider particles described by their space homogeneous distribution density \linebreak$f = f(t,v)$. We hence study the so-called spatially homogeneous Boltzmann equation:
\beqn \label{eq:Bol}
\partial_t f(t,v) = Q(f,f)(t,v), \quad v \in \R^3, \quad t \geq 0.
\eeqn
The Boltzmann collision operator is defined as
$$
Q(g,f):=\int_{\R^3 \times \Sp^2} B(v-v_*,\sigma) \left[g'_*f' - g_* f \right] \, d\sigma \, dv_*.
$$ 
Here and below, we are using the shorthand notations $f=f(v)$, $g_*=g(v_*)$, $f'=f(v')$ and $g'_*=g(v'_*)$. In this expression, $v$, $v_*$ and $v'$, $v'_*$ are the velocities of a pair of particles before and after collision. We make a choice of parametrization of the set of solutions to the conservation of momentum and energy (physical law of elastic collisions): 
$$
v+v_*=v'+v'_* ,
$$
$$
|v|^2+|v_*|^2=|v'|^2+ |v'_*|^2,
$$
so that the post-collisional velocities are given by:
$$
v'=\frac{v+v_*}{2} + \frac{|v-v_*|}{2} \sigma, \quad v'_*=\frac{v+v_*}{2} - \frac{|v-v_*|}{2} \sigma, \quad \sigma \in \Sp^2.
$$
The Boltzmann collision kernel $B(v-v_*,\sigma)$ only depends on the relative velocity $|v-v_*|$ and on the deviation angle $\theta$ through $\cos \theta = \langle \kappa, \sigma \rangle$ where $\kappa = (v-v_*)/|v-v_*|$ and $\langle \cdot, \cdot \rangle$ is the usual scalar product in $\R^3$. By a symmetry argument, one can always reduce to the case where $B(v-v_*, \sigma)$ is supported on $ \langle \kappa, \sigma \rangle \geq 0$ i.e $0 \leq \theta \leq \pi/2$. So, without loss of generality, we make this assumption. 

In this paper, we shall be concerned with the case when the kernel $B$ satisfies the following conditions:
\begin{itemize}
\item it takes product form in its arguments as
\beqn \label{eq:product}
B(v-v_*,\sigma) = \Phi(|v-v_*|) \, b(\cos \theta);
\eeqn
\item the angular function $b$ is locally smooth, and has a nonintegrable singularity for $\theta \rightarrow 0$, it satisfies for some $c_b>0$ and $s \in (0,1/2)$ (moderate angular singularity)
\beqn \label{eq:angularsing}
\forall \, \theta \in (0, \pi/2], \quad \frac{c_b}{\theta^{1+2s}} \leq \sin \theta \, b(\cos \theta) \leq \frac{1}{c_b \,  \theta^{1+2s}};
\eeqn
\item the kinetic factor $\Phi$ satisfies 
\beqn \label{eq:Phi}
\Phi(|v-v_*|)= |v-v_*|^\gamma \quad \text{with} \quad \gamma \in (0,1),
\eeqn
this assumption could be relaxed to assuming only that $\Phi$ satisfies $\Phi(\cdot) = C_\Phi \, |\cdot|^\gamma$ for some $\C_\Phi>0$. 
\end{itemize}

Our main physical motivation comes from particles interacting according to a repulsive potential of the form
\beqn \label{eq:potential}
\phi(r)=r^{-(p-1)}, \quad p \in (2,+\infty).
\eeqn

The assumptions made on $B$ throughout the paper include the case of potentials of the form~(\ref{eq:potential}) with $p>5$. Indeed, for repulsive potentials of the form~(\ref{eq:potential}), the collision kernel cannot be computed explicitly but Maxwell~\cite{Max} has shown that the collision kernel can be computed in terms of the interaction potential~$\phi$. More precisely, it satisfies the previous conditions (\ref{eq:product}), (\ref{eq:angularsing}) and (\ref{eq:Phi}) in dimension~$3$ (see~\cite{Cer,CIP,Vill2}) with $ s:=\frac{1}{p-1} \in (0,1)$ and $\gamma:= \frac{p-5}{p-1} \in (-3,1)$. 

One traditionally calls \emph{hard potentials} the case $p>5$ (for which $0<\gamma<1$), \emph{Maxwell molecules} the case $p=5$ (for which $\gamma=0$) and \emph{soft potentials} the case $2<p<5$ (for which $-3 < \gamma <0$). We can hence deduce that our assumptions made on $B$ include the case of hard potentials. 
\smallskip
The equation~(\ref{eq:Bol}) preserves mass, momentum and energy. Indeed, at least formally, we have:
$$
\int_{\R^3} Q(f,f)(v) \, \varphi(v) \, dv = 0 \quad \text{for} \quad \varphi(v) = 1,  v,  |v|^2;
$$
from which we deduce that a solution $f_t$ to the equation (\ref{eq:Bol}) is conservative, meaning that
\beqn \label{eq:conserv}
\forall \, t \geq0, \quad \int_{\R^3} f(t,v) \, \varphi(v) \, dv = \int_{\R^3} f_0(v) \, \varphi(v) \, dv \quad \text{for} \quad \varphi(v) = 1,  v,  |v|^2.
\eeqn

We introduce the entropy $H(f) = \int_{\R^3} f \, \log(f)$ and the entropy production $D(f)$.
Boltzmann's $H$ theorem asserts that 
\beqn \label{eq:entroprod}
\frac{d}{dt} H(f) = -D(f) \leq 0
\eeqn
and states that any equilibrium (i.e any distribution which maximizes the entropy) is a Maxwellian distribution $\mu_{\rho,u,T}$ for some $\rho>0$, $u \in \R^3$ and $T>0$:
$$
\mu^f(v) = \mu_{\rho,u,T}(v) = \frac{\rho \, e^{-\frac{|v-u|^2}{2T}}}{(2 \pi T)^{3/2}},
$$
where $\rho$, $u$ and $T$ are the mass, momentum and temperature of the gas:
$$
\rho := \rho_f = \int_{\R^3} f(v) \, dv, \quad u:= u_f = \frac{1}{\rho} \int_{\R^3} f(v) \, v \, dv, \quad T:= T_f = \frac{1}{3 \rho} \int_{\R^3} f(v) \, |v-u|^2 \, dv.
$$

Thanks to the conservation properties of the equation~(\ref{eq:conserv}), the following equalities hold:
$$
\rho_f = \rho_{f_0} , \quad u_f= u_{f_0}, \quad T_f= T_{f_0}
$$
where $f_0$ is the initial datum.

Moreover, a solution $f_t$ of the Boltzmann equation is expected to converge towards the Maxwellian distribution $\mu_{\rho,u,T}$ when $t \rightarrow + \infty$.

In this paper, we only consider the case of an initial datum satisfying
\beqn \label{eq:initialdatum}
\rho_{f_0}=1, \quad u_{f_0} = 0, \quad T_{f_0}=1,
\eeqn
one can always reduce to this situation (see~\cite{Vill1}). We then denote $\mu$ the Maxwellian with same mass, momentum and energy of $f_0$: $\mu(v)=(2 \pi)^{-3/2} e^{-|v|^2/2}$. 

\medskip
\subsection{Function spaces and notations}
We introduce some notations about weighted $L^p$ spaces. For some given Borel weight function $m \geq 0$ on $\R^3$, we define the Lebesgue weighted space $L^p(m)$, $1 \le p \le + \infty$, as the Lebesgue space associated to the norm
$$
\|h\|_{L^p(m)} := \|h\,m\|_{L^p}.
$$
We also define the weighted Sobolev space $W^{s,p}(m)$, $s \in \N$, $1 \le p < + \infty$, as the Sobolev space associated to the norm
$$
\|h\|_{W^{s,p}(m)} := \left( \sum_{|\alpha| \le s} \| \partial^\alpha h \|^p_{L^p(m)} \right)^{1/p}.
$$

Throughout this paper, we will use the same notation $C$ for positive constants that may change from line to line. Moreover, the notation $A \approx B$ will mean that there exist two constants $C_1$, $C_2>0$ such that $C_1 \, A \leq B \leq  C_2 \, A$.

\medskip
\subsection{Main results and known results}

\subsubsection{Convergence to equilibrium}
We first state our main result on exponential convergence to equilibrium.
\begin{theo} \label{th:main}
Consider a collision kernel $B$ satisfying conditions (\ref{eq:product}), (\ref{eq:angularsing}), (\ref{eq:Phi}) and $f_0$ a nonnegative distribution with finite mass, energy and entropy:
$$
f_0 \geq 0, \quad \int_{\R^3} f_0(v) \, (1+|v|^2) \, dv < \infty, \quad \int_{\R^3} f_0(v) \, \log(f_0(v)) \, dv < \infty  
$$
and satisfying (\ref{eq:initialdatum}). Then, if $f_t$ is a smooth solution (see Definition~\ref{def:sol}) to the equation~(\ref{eq:Bol}) with initial datum $f_0$, there exists a constant $C>0$ such that
$$
\forall \, t \geq 0, \quad \|f_t - \mu\|_{L^1} \leq C \, e^{-\lambda \, t}
$$
where $\lambda>0$ is defined in Theorem~\ref{th:mainlinear}.
\end{theo}

We improve a polynomial result of Villani~\cite{Vill1} and generalize to our context similar exponential results known for simplified models. Mouhot in~\cite{Mou2} proved such a result for the spatially homogeneous Boltzmann equation with hard potentials and Grad's cut-off. Carrapatoso in~\cite{Carr} recently proved exponential decay to equilibrium for the homogeneous Landau equation with hard potentials which is the grazing collisions limit of the model we study in the present paper. Let us also mention the paper of Gualdani et al.~\cite{GMM} where an exponential decay to equilibrium is proved for the inhomogeneous Boltzmann equation for hard spheres (see also~\cite{MM1,MisS,MM2} for related works).  

\smallskip

It is a known fact that our equation~(\ref{eq:Bol}) admits solutions which are conservative and satisfy some suitable properties of regularity, we will call them \emph{smooth solutions}. We here precise the meaning of this term and give an overview of results on the Cauchy theory of our equation. 
\begin{defin} \label{def:sol}
Let $f_0$ be a nonnegative function defined on $\R^3$ with finite mass, energy and entropy. We shall say that $(t,v) \mapsto f(t,v)$ is a \emph{smooth solution} to the equation~(\ref{eq:Bol}) if the following conditions are fulfilled:
\begin{itemize}
\item $f \geq 0$, $f \in \CC(\R^+, L^1)$;
\item for any $t \geq 0$,
$$
\int_{\R^3} f(t,v) \, \varphi(v) \, dv = \int_{\R^3} f_0(v) \, \varphi(v) \, dv \quad \text{for} \quad \varphi(v) = 1,  v,  |v|^2
$$
and
$$
\int_{\R^3} f(t,v) \, \log(f(t,v)) \,dv  + \int_0^t D(f(s,\cdot)) \, ds \leq  \int_{\R^3} f_0(v) \, \log(f_0(v)) \, dv \,;  
$$
where $D(f)$ is the entropy production defined in~(\ref{eq:entroprod});
\item for any $\varphi \in \CC^1(\R^+, \DDD(\R^3))$ and for any $t\geq 0$,
$$
\begin{aligned}
\int_{\R^3} f(t,v) \, \varphi(t,v) \, dv 
&= \int_{\R^3} f_0(v) \, \varphi(0,v) \, dv + \int_0^t \int_{\R^3} f(\tau, v) \, \partial_t \varphi(\tau,v)  \, dv \, d\tau \\
&\quad + \int_0^t \int_{\R^3} Q(f,f) (\tau, v) \, \varphi(\tau, v) \, dv \, d\tau 
\end{aligned}
$$
where the last integral is define through the following formula
$$
\int_{\R^3} Q(f,f)(v) \, \varphi(v) \, dv = \frac{1}{4} \int_{\R^3 \times \R^3 \times \Sp^2} B(v-v_*,\sigma) \, [f'_* f' - f_* f] \, (\varphi + \varphi_* - \varphi' - \varphi'_*) \, d\sigma \, dv_* \, dv;
$$
\item for any $t_0>0$ and for any $\ell \in \R^+$,
\beqn \label{eq:momentprod}
\sup_{t \geq t_0} \|f(t, \cdot)\|_{L^1(\langle v \rangle^\ell)} < \infty \,;
\eeqn
\item for any $t_0>0$ and for any $N$, $\ell \in \R^+$,
\beqn \label{eq:sobolevprod}
\sup_{t \geq t_0} \|f(t, \cdot)\|_{H^N(\langle v \rangle^\ell)} < \infty .
\eeqn
\end{itemize}
\end{defin}

Such a solution is known to exist. The problem of existence of solutions was first studied by Arkeryd in~\cite{Ark} where existence of solutions is proven for not too soft potentials, that is $\gamma >-1$ (Goudon~\cite{Gou} and Villani~\cite{Vill3} then improved this result enlarging the class of $\gamma$ considered). We mention that uniqueness of solution for hard potentials can be proved under some more restrictive conditions on the initial datum, see the paper of Desvillettes and Mouhot~\cite{DM} where $f_0$ is supposed to be regular ($f_0 \in W^{1,1}(\langle v \rangle^2)$) and the paper of Fournier and Mouhot~\cite{FM} where $f_0$ is supposed to be localized ($\int_{\R^3} f_0 \, e^{a|v|^\gamma}  dv < \infty$ for some $a>0$) for hard potentials.

Concerning the moment production property, it was discovered by Elmroth~\cite{Elm} and Desvillettes~\cite{Desv1} and improved by Wennberg~\cite{Wenn}, which justifies our point~(\ref{eq:momentprod}) in the definition of a smooth solution. We here point out the fact that this property is not anymore true for Maxwell molecules or soft potentials. As a consequence, our method, which relies partially on this property, works only for hard potentials.

Finally,  we mention papers where regularization results are proven for ``true" (that is non mollified) physical potentials: \cite{ADVW} by Alexandre et al. and \cite{CH} by Chen and He where the initial datum is supposed to have finite energy and entropy, \cite{BF} by Bally and Fournier where only the $2D$ case is treated and \cite{Four} by Fournier under others conditions on the initial datum. Theorem~1.4 from \cite{CH} explains our point~(\ref{eq:sobolevprod}). 



\smallskip

We now recall previous results on convergence to equilibrium for solutions to equation~(\ref{eq:Bol}). It was first studied by Carlen and Carvalho \cite{CC1,CC2} and then by Toscani and Villani \cite{ToscVill2}. Up to now, the best rate of convergence in our case was obtained by Villani in~\cite{Vill1}:
\begin{theo} \label{th:polydecay}
Let us consider $f_t$ a smooth solution to (\ref{eq:Bol}) with an initial datum $f_0$ satisfying (\ref{eq:initialdatum}) with finite entropy. Then $f_t$ satisfies the following polynomial decay to equilibrium: for any $t_0>0$ and any $\varepsilon>0$, there exists $C_{t_0,\varepsilon}>0$ such that 
$$
\forall \, t \geq t_0, \quad \|f_t - \mu\|_{L^1} \le C_{t_0,\varepsilon} \, t^{-\frac{1}{ \varepsilon}}.
$$
\end{theo}
This result comes from~\cite[Theorem~4.1]{Vill1} which states that if $f$ is a function which satisfies the following lowerbound
\beqn \label{eq:lowerbound}
\forall \, v \in \R^3, \quad f(v) \geq K_0 \, e^{-A_0|v|^{q_0}} \quad \text{with} \quad K_0, A_0 >0, \, q_0 \geq 2
\eeqn
then for any $\varepsilon>0$, there exists an explicit constant $K_\varepsilon>0$ such that 
\beqn \label{eq:entropy}
D(f) \geq K_\varepsilon \, H(f | \mu)^{1+\varepsilon}.
\eeqn
 
It is a result from Mouhot~\cite[Theorem 1.2]{Mou1} that the lowerbound~(\ref{eq:lowerbound}) holds for any smooth solution $f_t$ of our equation~(\ref{eq:Bol}). Let us mention that lowerbounds of solutions were first studied by Carleman~\cite{Carl} (for hard spheres) and then by Pulvirenti and Wennberg~\cite{PulWenn} (for hard potentials with cut-off). Finally, Mouhot~\cite{Mou1} extended these results to the spatially inhomogeneous case without cut-off. We here state Theorem 1.2 from~\cite{Mou1} that we use:
for any $t_0>0$ and for any exponent $q_0$ such that 
$$
q_0>2 \frac{\log\left(2 + \frac{2s}{1-s} \right)}{\log 2},
$$
a smooth solution $f_t$ to~(\ref{eq:Bol}) satisfies 
$$
\forall \, t \geq t_0, \quad \forall \, v  \in \R^3, \quad f(t,v) \geq K_0 \, e^{-A_0 |v|^{q_0}}
$$
for some $K_0$, $A_0 >0$. 

We can then deduce that the conclusion of Theorem~\ref{th:polydecay} holds using the Csisz\'{a}r-Kullback-Pinsker inequality $\|f-\mu\|_{L^1} \leq \sqrt{2 H(f|\mu)}$ combined with the result of Villani~(\ref{eq:entropy}).

\smallskip
Let us here emphasize that the method of Villani to prove the polynomial convergence towards equilibrium is purely nonlinear. Ours is based on the study of the linearized equation. 
 
\smallskip
\subsubsection{The linearized equation}
We introduce the linearized operator. Considering the linearization $f=\mu+h$, we obtain at first order the linearized equation around the equilibrium~$\mu$
\beqn \label{eq:linearop}
\partial_t h = \LL h := Q(\mu,h)+ Q(h,\mu),
\eeqn
for $h=h(t,v)$, $v \in \R^3$. The null space of the operator $\LL$ is the $5$-dimensional space
\beqn \label{eq:null}
\NN(\LL) = \text{Span} \left\{ \mu, \mu \, v_1, \mu \, v_2, \mu \, v_3, \mu \,  |v|^2 \right\}.
\eeqn

Our strategy is to combine the polynomial convergence to equilibrium and a spectral gap estimate on the linearized operator to show that if the solution enters some stability neighborhood of the equilibrium, then the convergence is exponential in time. Previous results on spectral gap estimates hold only in $L^2(\mu^{-1/2})$ and the Cauchy theory for the nonlinear Boltzmann equation is constructed in $L^1$-spaces with polynomial weight. In order to link the linear and the nonlinear theories, our approach consists in proving new spectral gap estimates for the linearized operator $\LL$ in spaces of type $L^1(\langle v \rangle^k)$. To do that, we exhibit a convenient splitting of the linearized operator in such a way that we may use the abstract theorem from~\cite{GMM} which allows us to enlarge the space of spectral estimates of a given operator. 

Here is the result we obtain on the linearized equation which provides a constructive spectral gap estimate for $\LL$ in $L^1(\langle v \rangle^k)$ and which is the cornerstone of the proof of Theorem~\ref{th:main}.

\smallskip
\begin{theo} \label{th:mainlinear}
Let $k>2$ and a collision kernel $B$ satisfying (\ref{eq:product}), (\ref{eq:angularsing}) and (\ref{eq:Phi}). Consider the linearized Boltzmann operator $\LL$ defined in~(\ref{eq:linearop}). Then for any positive $\lambda < \min(\lambda_0,\lambda_k)$ (where $\lambda_0$ is the spectral gap of $\LL$ in $L^2(\mu^{-1/2})$ defined in Proposition~\ref{prop:L2} and $\lambda_k$ is a constant depending on $k$ defined in Lemma~\ref{lem:Bdissip}), there exists an explicit constant $C_\lambda >0$, such that for any $h \in L^1(\langle v \rangle^k)$, we have the following estimate
\beqn \label{eq:sgdecay}
\forall \, t \geq 0, \quad \|S_\LL(t) h -  \Pi h \|_{L^1(\langle v \rangle^k)} \leq C_\lambda \, e^{-\lambda t } \|h - \Pi h\|_{L^1(\langle v \rangle^k)},
\eeqn
where $S_\LL(t)$ denotes the semigroup of $\LL$ and $\Pi$ the projection onto $\NN(\LL)$. 
\end{theo}

Let us briefly review the existing results concerning spectral gap estimates for $\LL$. Pao~\cite{Pao} studied spectral properties of the linearized operator for hard potentials by non-constructive and very technical means. This article was reviewed by Klaus~\cite{Klaus}. Then, Baranger and Mouhot gave the first explicit estimate on this spectral gap in~\cite{BM} for hard potentials ($\gamma>0$). If we denote $\DD$ the Dirichlet form associated to $-\LL$:
$$
\DD(h):=\int_{\R^3} (-\LL h) \, h \, \mu^{-1},
$$
and $\NN(\LL)^\perp$ the orthogonal of $\NN(\LL)$ defined in~(\ref{eq:null}) and $\Pi$ the projection onto $\NN(\LL)$, the Dirichlet form $\DD$ satisfies 
\beqn \label{eq:entropyL2}
\forall \, h \in \NN(\LL)^\perp, \quad \DD(h) \geq \lambda_0 \, \|h\|^2_{L^2(\mu^{-1/2})},
\eeqn
for some constructive constant $\lambda_0>0$. This result was then improved by Mouhot~\cite{Mou3} and later by Mouhot and Strain~\cite{MS}. In the last paper, it was conjectured that a spectral gap exists if and only if $\gamma+2s \geq 0$. This conjecture was finally proven by Gressman and Strain in~\cite{GS}. 


\smallskip

Another question would be to obtain similar results in other spaces: $L^p$ spaces with $1 < p \leq 2$ and a polynomial weight or $L^p$ spaces with $1 \leq p \leq 2$ and a stretched exponential weight. Our computations do not allow to conclude in those cases, more precisely, we are not able to do the computations which allow to obtain the suitable splitting of the linear operator in order to apply the theorem of enlargement of the space of spectral estimates. As a consequence, we can not prove the existence of a spectral gap on those spaces. However, we believe that such results may hold.

We here point out that the knowledge of a spectral gap estimate in $L^1(\langle v \rangle^k)$ for the fractional Fokker-Planck equation (see~\cite{Tris}) is consistent with our result.
Indeed, the behavior of the Boltzmann collision operator has been widely conjectured to be that of a fractional diffusion (see~\cite{Desv4,Gou,Vill3}).

\medskip
\noindent\textbf{Acknowledgments.} We thank St\'{e}phane Mischler for fruitful discussions and his encouragement.

\bigskip

\section{The linearized equation} 
\label{sec:lin}
\setcounter{equation}{0}
\setcounter{theo}{0}



Here and below, we denote $m(v):=\langle v \rangle ^k$ with $k>2$. The aim of the present section is to prove Theorem~\ref{th:mainlinear}. To do that, we exhibit a splitting of the linearized operator into two parts, one which is bounded and the second one which is dissipative. We can then apply the abstract theorem of enlargement of the functional space of the semigroup decay from Gualdani et al.~\cite{GMM} (see Subsection~\ref{subsec:L1}).

\medskip

\subsection{Notations}

We now introduce notations about spectral theory of unbounded operators. For a given real number $a \in \R$, we define the half complex plane
$$
\Delta_a := \left\{ z \in \C, \, \Re e \, z > a \right\}.
$$

\smallskip
For some given Banach spaces $(E,\|\cdot \|_E)$ and $(\EE,\| \cdot
\|_\EE)$, we denote by $\mathscr{B}(E, \EE)$ the space of bounded linear
operators from $E$ to $\EE$ and we denote by $\| \cdot
\|_{\mathscr{B}(E,\EE)}$ or $\| \cdot \|_{E \to \EE}$ the associated norm
operator. We write $\mathscr{B}(E) = \mathscr{B}(E,E)$ when $E=\EE$.
We denote by $\mathscr{C}(E,\EE)$ the space of closed unbounded linear
operators from $E$ to $\EE$ with dense domain, and $\mathscr{C}(E)=
\mathscr{C}(E,E)$ in the case $E=\EE$.

\smallskip
For a Banach space $X$ and   $\Lambda \in \mathscr{C}(X)$ we denote by $S_\Lambda(t)$, $t \ge
0$, its semigroup, by $\mbox{D}(\Lambda)$ its domain, by
$\mbox{N}(\Lambda)$ its null space and by $\mbox{R}(\Lambda)$ its range. We also
denote by $\Sigma(\Lambda)$ its spectrum, so that for any $z$ belonging to the resolvent set $\rho(\Lambda) :=  \C
\backslash \Sigma(\Lambda)$  the operator $\Lambda - z$ is invertible
and the resolvent operator
$$
\RR_\Lambda(z) := (\Lambda -z)^{-1}
$$
is well-defined, belongs to $\mathscr{B}(X)$ and has range equal to
$\mbox{D}(\Lambda)$.
An eigenvalue $\xi \in
\Sigma(\Lambda)$ is said to be isolated if
\[
\Sigma(\Lambda) \cap \left\{ z \in \C, \,\, |z - \xi| \le r \right\} =
\{ \xi \} \ \mbox{ for some } r >0.
\]
In the case when $\xi$ is an isolated eigenvalue, we may define
$\Pi_{\Lambda,\xi} \in \mathscr{B}(X)$ the associated spectral projector by
$$
\Pi_{\Lambda,\xi} := - {1 \over
  2i\pi} \int_{ |z - \xi| = r' } \RR_\Lambda (z) \, dz
$$
with $0<r'<r$. Note that this definition is independent of the value
of $r'$ as the application
$
\C \setminus \Sigma(\Lambda) \to \mathscr{B}(X)$, $z \to \RR_{\Lambda}(z)$ is holomorphic.
For any $\xi \in \Sigma(\Lambda)$ isolated, it is well-known (see~\cite{Kato} paragraph III-6.19) 
that $\Pi_{\Lambda,\xi}^2=\Pi_{\Lambda,\xi}$,  so that $\Pi_{\Lambda,\xi}$ is indeed a projector.

\smallskip
When moreover the so-called ``algebraic eigenspace" $\mbox{R}(\Pi_{\Lambda,\xi})$ is finite dimensional we say that
$\xi$ is a discrete eigenvalue, written as $\xi \in \Sigma_d(\Lambda)$. 

 \medskip

\subsection{Spectral gap in $L^2(\mu^{-1/2})$} \label{subsec:L2} 
We here state a direct consequence of inequality~(\ref{eq:entropyL2}) from~\cite{BM}, which gives us a spectral gap estimate in $L^2(\mu^{-1/2})$. 
\begin{prop} \label{prop:L2}
There is a constructive constant $\lambda_0>0$ such that 
$$
\forall \, t \geq 0, \quad \forall \, h \in L^2(\mu^{-1/2}), \quad \|S_\LL(t) h-  \Pi h\|_{L^2(\mu^{-1/2})} \leq e^{-\lambda_0 t} \|h - \Pi h\|_{L^2(\mu^{-1/2})}.
$$
\end{prop}

\medskip
\subsection{Splitting of the linearized operator}

We first split the linearized operator $\LL$ defined in~(\ref{eq:linearop}) into two parts, separating the grazing collisions and the cut-off part, we define 
$$
b_\delta := \mathds{1}_{\theta \leq \delta} \, b \quad  \text{and} \quad b^c_\delta := \mathds{1}_{\theta \geq \delta} \, b
$$ 
for some $\delta\in (0,1)$ to be chosen later, it induces the following splitting of $\LL$:
$$
\begin{aligned}
\LL h&= \LL_\delta h+ \LL^c_\delta h \\
&=: \int_{\R^3 \times \Sp^2} \left[\mu'_* h' - \mu_* h + h'_* \mu' - h_* \mu \right] \, b_\delta(\cos \theta) |v-v_*|^\gamma \, d\sigma \, dv_* \\
&\quad +  \int_{\R^3 \times \Sp^2} \left[\mu'_* h' - \mu_* h + h'_* \mu' - h_* \mu \right] \, b^c_\delta(\cos \theta) |v-v_*|^\gamma \, d\sigma \, dv_*.
\end{aligned}
$$
In the rest of the paper, we shall use the notations 
$$
B_\delta(v-v*, \sigma) := b_\delta(\cos \theta) \, |v-v_*|^\gamma \quad \text{and} \quad B^c_\delta(v-v_*, \sigma) := b^c_\delta(\cos \theta)\, |v-v_*|^\gamma.
$$

As far as the cut-off part is concerned, our strategy is similar as the one adopted in~\cite{GMM} for hard-spheres. For any $\varepsilon \in (0,1)$, we consider $\Theta_\varepsilon \in \CC^\infty$ bounded by one, which equals one on 
$$
\left\{|v| \le \varepsilon^{-1} \text{ and } 2\varepsilon \le |v-v_*| \le \varepsilon^{-1} \text{ and } | \cos \, \theta | \le 1- 2\varepsilon \right\}
$$
and whose support is included in
$$
\left\{|v| \le 2\varepsilon^{-1} \text{ and } \varepsilon \le |v-v_*| \le 2\varepsilon^{-1} \text{ and } | \cos \, \theta | \le 1- \varepsilon \right\}.
$$
We then denote the truncated operator
$$
\AA_{\delta,\varepsilon} (h) :=   \int_{\R^3 \times \Sp^2} \! \Theta_\varepsilon \! \left[\mu'_*\,  h' \, + \, \mu' \, h'_*- \mu \,  h_* \right] \, b^c_\delta(\cos \theta) |v-v_*|^\gamma  \, d\sigma \, dv_*
$$
and the corresponding remainder operator
$$
\BB^{c}_{\delta,\varepsilon} (h) := \int_{\R^3 \times \Sp^2} \! (1 -\Theta_\varepsilon) \! \left[\mu'_*\,  h' \, + \, \mu' \, h'_*- \mu \,  h_* \right] \, b^c_\delta(\cos \theta) |v-v_*|^\gamma  \, d\sigma \, dv_*.
$$
We also introduce 
$$
\nu_\delta(v) :=  \int_{\R^3 \times \Sp^2}  \mu_* \, b^c_\delta(\cos \theta) |v-v_*|^\gamma  \, d\sigma \, dv_*,
$$ 
so that we have the following splitting: $\LL_\delta^c = \AA_{\delta,\varepsilon} + \BB_{\delta,\varepsilon}^c - \nu_\delta$.

Moreover, $\nu_\delta$ satisfies 
$$
\nu_\delta(v) = K_\delta \,(\mu \ast | \cdot|^\gamma)(v) 
$$
with
$$
 K_\delta := \int_{\Sp^2} b^c_\delta(\cos \theta) \, d\sigma \approx \int_\delta^{\pi/2} b(\cos \theta) \, \sin \theta \, d\theta \approx \delta^{-2s} - \left(\frac{\pi}{2}\right)^{-2s} \xrightarrow[\delta \rightarrow 0]{} +\infty 
$$
using the spherical coordinates to get the second equality and~(\ref{eq:angularsing}) to get the final one; and
$$
(\mu \ast | \cdot|^\gamma)(v) \approx \langle v \rangle^\gamma.
$$

We finally define 
$$
 \BB_{\delta,\varepsilon} := \LL_\delta + \BB^c_{\delta,\varepsilon} - \nu_\delta
$$
so that $\LL = \AA_{\delta,\varepsilon} + \BB_{\delta,\varepsilon}$.

\smallskip
\subsubsection{Dissipativity properties} \label{subsubsec:dissip}
\begin{lem} \label{lem:Ldeltadissip}
There exists a function $\varphi_k(\delta)$ depending on $k$ and tending to $0$ as $\delta$ tends to $0$ such that for any $h \in L^1(\langle v \rangle^\gamma m)$, the following estimate holds:
\beqn \label{eq:Ldeltadissip}
\int_{\R^3} \LL_\delta (h) \, \mathrm{sign} (h) \, m \, dv \leq \varphi_k(\delta) \, \|h\|_{L^1(\langle v \rangle^\gamma m)}.
\eeqn
\end{lem}

\begin{proof}
Let us first introduce a notation which is going to be useful in the sequel of the proof:
\beqn \label{eq:kappa}
\kappa_\delta := \int_0^{\pi/2} b_\delta (\cos \theta) \, \sin^2(\theta) \, d\theta = \int_0^\delta b(\cos\theta) \, \sin^2(\theta) \, d\theta \approx  \delta^{1-2s}  \xrightarrow[\delta \rightarrow 0]{} 0,
\eeqn
where the last equality comes from~(\ref{eq:angularsing}). We here underline the fact that considering a moderate singularity, meaning $s \in (0,1/2)$, is here needed to get the convergence of $\kappa_\delta$ to~$0$ as $\delta$ goes to $0$. 

\smallskip
We split $\LL_\delta$ into two parts in the following way:
$$
\begin{aligned}
\LL_\delta h &= \int_{\R^3 \times \Sp^2} \left[\mu'_*\,  h' \, - \mu_* \,  h \right] \, b_\delta(\cos \theta) |v-v_*|^\gamma  \, d\sigma \, dv_* \\
&\quad + \int_{\R^3 \times \Sp^2} \left[h'_*\,  \mu' \, - h_* \,  \mu \right] \, b_\delta(\cos \theta) |v-v_*|^\gamma  \, d\sigma \, dv_*  \\
&=: \LL^1_\delta h + \LL^2_\delta h,
\end{aligned}
$$
this splitting corresponds to the splitting of $\LL_\delta$ as $Q_\delta(\mu,h)+Q_\delta(h,\mu)$ if $Q_\delta$ denotes the collisional operator associated to the kernel $B_\delta$. 

\smallskip
We first deal with $\LL^1_\delta$. Let us recall that we have $\mu \, \mu_* = \mu' \, \mu'_*$. In the following computation, we denote $g:=h \, \mu^{-1}$:
$$
\begin{aligned}
& \quad \int_{\R^3} \LL^1_\delta (h) \, \text{sign} (h) \, m \, dv \\
&= \int_{\R^3 \times \R^3 \times \Sp^2} B_\delta (v-v_*,\sigma)  \, \mu \, \mu_* \, \left[g' - g \right] \, \text{sign}(g) \, m \, d\sigma \, dv_* \, dv \\
&= \int_{\R^3 \times \R^3 \times \Sp^2} B_\delta (v-v_*,\sigma)  \, \mu \, \mu_* \, \left[g' - g \right] \, \left[\text{sign}(g) - \text{sign}(g')\right] \, m \, d\sigma \, dv_* \, dv \\
&\quad + \int_{\R^3 \times \R^3 \times \Sp^2} B_\delta (v-v_*,\sigma) \, \mu \, \mu_* \, \left[g' - g \right] \, \text{sign}(g') \, m \, d\sigma \, dv_* \, dv \\
&\leq \int_{\R^3 \times \R^3 \times \Sp^2} B_\delta (v-v_*,\sigma) \, \mu \, \mu_* \, \left[g' - g \right] \, \text{sign}(g') \, m \, d\sigma \, dv_* \, dv,
\end{aligned}
$$
where we used that for any $a$, $b \in \R$, $(a-b)(\text{sign}(a)-\text{sign}(b)) \leq 0$ to get the last inequality. 
\begin{rem}
We here emphasize that this computation is particularly convenient in the $L^1$ case since $\text{sign}(h) = \text{sign}(g)$. In the $L^p$ case, it is trickier and for now, we are not able to adapt it to get the wanted estimates. 
\end{rem}

We now use the classical pre-post collisional change of variables to pursue the computation:
$$
\begin{aligned}
\int_{\R^3} \LL^1_\delta (h) \, \text{sign} (h) \, m \, dv
&\leq \int_{\R^3 \times \R^3 \times \Sp^2} B_\delta (v-v_*,\sigma) \, \mu \, \mu_* \, \left[g - g' \right] \, \text{sign}(g) \, m' \, d\sigma \, dv_* \, dv \\
&= \int_{\R^3 \times \R^3 \times \Sp^2} B_\delta (v-v_*,\sigma) \, \mu \, \mu_* \, \left[g - g' \right] \, \text{sign}(g) \, (m' - m) \, d\sigma \, dv_* \, dv \\
&\quad + \int_{\R^3 \times \R^3 \times \Sp^2} B_\delta (v-v_*,\sigma) \, \mu \, \mu_* \, \left[g - g' \right] \, \text{sign}(g) \, m \, d\sigma \, dv_* \, dv. \\
&= \int_{\R^3 \times \R^3 \times \Sp^2} B_\delta (v-v_*,\sigma) \, \mu \, \mu_* \, \left[g - g' \right] \, \text{sign}(g) \, (m' - m) \, d\sigma \, dv_* \, dv \\
&\quad - \int_{\R^3} \LL^1_\delta (h) \, \text{sign} (h) \, m \, dv.
\end{aligned}
$$
We hence deduce that 
$$
\begin{aligned}
&\quad \int_{\R^3} \LL^1_\delta (h) \, \text{sign} (h) \, m \, dv \\
&\leq \frac{1}{2} \int_{\R^3 \times \R^3 \times \Sp^2} B_\delta (v-v_*,\sigma) \, \mu \, \mu_* \, \left[g - g' \right] \, \text{sign}(g) \, (m' - m) \, d\sigma \, dv_* \, dv \\
&\leq \int_{\R^3 \times \R^3 \times \Sp^2} B_\delta (v-v_*,\sigma) \, \mu \, \mu_* \, |g| \, |m'-m|  \, d\sigma \, dv_* \, dv \\
&= \int_{\R^3 \times \R^3 \times \Sp^2} B_\delta (v-v_*,\sigma) \, \mu_* \, |h| \, |m'-m|  \, d\sigma \, dv_* \, dv.
\end{aligned}
$$
We now estimate the difference $|m'-m|$:
$$
|m'-m| \leq \left(\sup_{z \in \text{B}(v, |v'-v|)} \left|\nabla m\right|(z) \right) \, |v'-v|, 
$$
with
$$
 |v'-v| = |v-v_*|/2 \, \sin \left( \theta/2 \right) \leq \frac{1}{2 \sqrt{2}} \, |v-v_*| \, \sin \theta.
 $$
Then, we use the fact 
$$ 
\begin{aligned}
\sup_{z \in \text{B}(v, |v'-v|)} \left|\nabla m\right|(z)   
&\leq k \, 2^{k-1} \, \left(\langle v \rangle ^{k-2} + \langle v-v' \rangle^{k-1}\right) \\
&\leq k \, 2^{2(k-1)} \, \left(\langle v \rangle ^{k-2} + \langle v_* \rangle^{k-1}\right),
\end{aligned}
$$
which implies that 
\beqn \label{eq:gradm}
|m'-m| \leq C_k \, |v-v_*| \, \sin \theta \, \left(\langle v \rangle ^{k-1} + \langle v_* \rangle^{k-1}\right),
\eeqn
for some constant $C_k>0$ depending on $k$.
\begin{rem}
We here point out that this kind of estimate does not hold in the case of a stretched exponential weight. Indeed, taking the gradient of a stretched exponential function, there is not anymore a gain in the degree as in the case of a polynomial function.
\end{rem}
We finally obtain
\beqn \label{eq:L1delta}
\begin{aligned}
&\quad \int_{\R^3} \LL^1_\delta (h) \, \text{sign} (h) \, m \, dv\\
&\leq C_k \, \int_{\R^3 \times \R^3 \times \Sp^2} b_\delta(\cos \theta) \sin \theta \, \mu_* \, |v-v_*|^{\gamma +1} \left(\langle v \rangle ^{k-1} + \langle v_* \rangle^{k-1}\right) \, |h| \, d\sigma \, dv_* \, dv \\
&\leq C_k \, \int_0^{\pi/2} b_\delta(\cos\theta) \sin^2(\theta) \, d\theta \int_0^{2\pi} d\phi \int_{\R^3 \times \R^3} \mu_* |v-v_*|^{\gamma+1}  \left(\langle v \rangle ^{k-1} + \langle v_* \rangle^{k-1}\right) \, |h| \, dv_* \, dv \\
&\leq C_k \,  \, \kappa_\delta \, \int_{\R^3} |h| \, \langle v \rangle^\gamma \, m \, dv,
\end{aligned}
\eeqn
where we used spherical coordinates to obtain the second inequality and~(\ref{eq:kappa}) to obtain the last one.

\smallskip
We now deal with $\LL^2_\delta$. We split it into two parts:
$$
\begin{aligned}
\LL^2_\delta h &= \int_{\R^3 \times \Sp^2} \left[h'_*\,  \mu' \, - h_* \,  \mu \right] \, b_\delta(\cos \theta) |v-v_*|^\gamma  \, d\sigma \, dv_* \\
&= \int_{\R^3 \times \Sp^2} B_\delta(v-v_*,\sigma)  \, h'_*\, \left[\mu'-\mu \right] \,  d\sigma \, dv_* + \int_{\R^3 \times \Sp^2} B_\delta(v-v_*,\sigma)  \, \left[h'_*- h_* \right] \,  d\sigma \, dv_*  \, \mu \\
&=: \LL^{2,1}_\delta h + \LL^{2,2}_\delta h.
\end{aligned}
$$

Concerning $\LL^{2,2}_\delta$, we use the cancellation lemma~\cite[Lemma~1]{ADVW}. It implies that 
$$
\LL^{2,2}_\delta h = \left(S_\delta \ast h \right) \, \mu
$$
with 
$$
\begin{aligned}
S_\delta(z) &= 2 \pi \, \int_0^{\pi/2} \sin\theta \, b_\delta(\cos \theta) \left( \frac{|z|^\gamma}{\cos^{\gamma+3} (\theta/2)} - |z|^\gamma \right) \, d\theta \\
&= 2 \pi \, |z|^\gamma \, \int_0^{\pi/2} \sin \theta  \, b_\delta(\cos \theta) \, \frac{1-\cos^{\gamma+3} ( \theta/2)}{\cos^{\gamma+3} (\theta/2)} \, d\theta \\
&= 2 \pi \, |z|^\gamma \, \int_0^{\delta} \sin \theta  \, b(\cos \theta) \, \frac{1-\cos^{\gamma+3} ( \theta/2)}{\cos^{\gamma+3} (\theta/2)} \, d\theta \\
&\leq C \, |z|^\gamma \, \int_0^\delta \sin \theta  \, b(\cos \theta) \, \theta^2 \, d\theta \\
& \leq C \, \delta^{2-2s} \, |z|^\gamma,
\end{aligned}
$$
where the next-to-last inequality comes from the fact that $\frac{1-\cos^{\gamma+3} ( \theta/2)}{cos^{\gamma+3} (\theta/2)} \sim \frac{\gamma+3}{2} \, \theta^2$ as $\theta$ goes to $0$. We hence deduce that  for any $\theta \in (0,\delta)$, $\frac{1-\cos^{\gamma+3} ( \theta/2)}{cos^{\gamma+3} (\theta/2)} \leq C \, \theta^2$ for some $C>0$; 
and the last inequality comes from~(\ref{eq:angularsing}).
We deduce that 
\beqn \label{eq:L22delta}
\begin{aligned}
 \int_{\R^3} \LL^{2,2}_\delta (h) \, \text{sign} (h) \, m \, dv
 &\leq \int_{\R^3} \left|S_\delta \ast h \right| \, m \, \mu \, dv \\
 &\leq C \, \delta^{2-2s} \, \int_{\R^3} \left(| \cdot |^\gamma \ast |h|\right) \, \mu \, m \, dv \\
 &\leq C \, \delta^{2-2s} \, \int_{\R^3} \left( | \cdot |^\gamma \ast \mu \, m \right) \, |h| \, dv \\
 &\leq C \, \delta^{2-2s} \, \int_{\R^3} |h| \, \langle v \rangle^\gamma \, dv.
 \end{aligned}
\eeqn

\smallskip
We now deal with $\LL^{2,1}_\delta$. To do that, we introduce the notation $M:=\sqrt{\mu}$ and write that $\mu' - \mu = (M' - M)(M' +M)$, which implies 
$$
\begin{aligned}
\int_{\R^3} \LL^{2,1}_\delta (h) \, \text{sign} (h) \, m \, dv
&\leq \int _{\R^3 \times \R^3 \times \Sp^2} B_\delta(v-v_*,\sigma) \, |h'_*| \, |M'-M| \, (M'+M) \, m \, d\sigma \, dv_* \, dv \\
&\leq  \int _{\R^3 \times \R^3 \times \Sp^2} B_\delta(v-v_*,\sigma) \, |h'_*| \, |M'-M| \, M' \, |m'-m| \, d\sigma \, dv_* \, dv \\
&\quad +  \int _{\R^3 \times \R^3 \times \Sp^2} B_\delta(v-v_*,\sigma) \, |h'_*| \, |M'-M| \, M' \, m' \, d\sigma \, dv_* \, dv \\
&\quad + \int _{\R^3 \times \R^3 \times \Sp^2} B_\delta(v-v_*,\sigma) \, |h'_*| \, |M'-M| \, M \, m \, d\sigma \, dv_* \, dv. 
\end{aligned}
$$
We now perform the pre-post collisional change of variables, which gives us:
$$
\begin{aligned} 
\int_{\R^3} \LL^{2,1}_\delta h \, \text{sign} (h) \, m \, dv
&\leq \int _{\R^3 \times \R^3 \times \Sp^2} B_\delta(v-v_*,\sigma) \, |h_*| \, |M'-M| \, M \, |m'-m| \, d\sigma \, dv_* \, dv \\
&\quad +  \int _{\R^3 \times \R^3 \times \Sp^2} B_\delta(v-v_*,\sigma) \, |h_*| \, |M'-M| \, M \, m \, d\sigma \, dv_* \, dv \\
&\quad + \int _{\R^3 \times \R^3 \times \Sp^2} B_\delta(v-v_*,\sigma) \, |h_*| \, |M'-M| \, M' \, m' \, d\sigma \, dv_* \, dv \\
&=: I_1 + I_2 + I_3.
\end{aligned}
$$ 

For the term $I_1$, we use the fact that $M$ is bounded and the estimate~(\ref{eq:gradm}) on $|m'-m|$:
\beqn \label{eq:I1}
\begin{aligned}
I_1 &\leq C_k  \int_{\R^3 \times \R^3 \times \Sp^2} b_\delta(\cos \theta) \, \sin \theta \, |v-v_*|^{\gamma+1} \, |h_*| \, M \, \left(\langle v \rangle^{k-1} + \langle v_* \rangle^{k-1} \right) \, d\sigma \, dv_* \, dv \\
&\leq C_k \, \kappa_\delta \int_{\R^3} |h| \, \langle v \rangle^\gamma \, m \, dv.
\end{aligned}
\eeqn

The term $I_2$ is treated using that $M$ is Lipschitz continuous, we obtain:
\beqn \label{eq:I2}
I_2 \leq C \, \kappa_\delta \int_{\R^3} \, |h| \, \langle v \rangle^{\gamma +1} \, dv.
\eeqn

To treat $I_3$, we first estimate the integral
$$
\int_{\R^3 \times \Sp^2} B_\delta(v-v_*, \sigma) \, |M'-M| \, M' \, m' \, d\sigma \, dv =: J(v_*)=J.
$$
Using the fact that $M$ is Lipschitz continuous, we have
$$
J \leq C  \int_{\R^3 \times \Sp^2} b_\delta(\cos \theta) \, \sin (\theta/2) \, |v-v_*|^{\gamma+1} \, M' \, m' \, d\sigma \, dv.
$$
Then, for each $\sigma$, with $v_*$ still fixed, we perform the change of variables $v \rightarrow v'$. This change of variables is well-defined on the set $\left\{ \cos \theta > 0 \right\}$. Its Jacobian determinant is 
$$
\left| \frac{dv'}{dv} \right| = \frac{1}{8} (1 + \kappa \cdot \sigma) = \frac{(\kappa' \cdot \sigma)^2}{4},
$$
where $\kappa=(v-v_*)/|v-v_*|$ and $\kappa'=(v'-v_*)/|v'-v_*|$. We have $\kappa' \cdot \sigma = \cos(\theta/2) \geq 1/\sqrt{2}$. The inverse transformation $v' \rightarrow \psi_\sigma(v')=v$ is then defined accordingly. Using the fact that 
$$
\cos \theta = \kappa \cdot \sigma = 2 (\kappa' \cdot \sigma)^2 -1 \quad \text{and} \quad \sin (\theta/2) = \sqrt{ 1 - \cos^2(\theta/2)} = \sqrt{1 - (\kappa' \cdot \sigma)^2},
$$
we obtain
$$
\begin{aligned}
J &\leq C  \int_{\R^3 \times \Sp^2} b_\delta(2(\kappa' \cdot \sigma)^2 -1) \, \sqrt{1 - (\kappa' \cdot \sigma)^2} \, |\psi_\sigma(v') - v_*|^{\gamma +1} \, M(v') \, m(v') \, dv \, d\sigma \\
&\leq C  \int_{\kappa' \cdot \sigma \geq 1/\sqrt{2}} b_\delta(2(\kappa' \cdot \sigma)^2 -1)  \sqrt{1 - (\kappa' \cdot \sigma)^2}  |\psi_\sigma(v') - v_*|^{\gamma +1}  M(v')  m(v')  \frac{1}{(\kappa' \cdot \sigma)^2}  dv' \, d\sigma \\
&\leq C  \int_{\kappa \cdot \sigma \geq 1/\sqrt{2}} b_\delta(2(\kappa \cdot \sigma)^2 -1) \, \sqrt{1 - (\kappa \cdot \sigma)^2} \, |\psi_\sigma(v) - v_*|^{\gamma +1} \, M(v) \, m(v) \, \frac{1}{(\kappa \cdot \sigma)^2} \, dv \, d\sigma. 
\end{aligned}
$$
We now use the fact that $|\psi_\sigma(v) - v_*| = |v-v_*|/(\kappa \cdot \sigma)$. We deduce that 
$$
\begin{aligned}
J &\leq C  \int_{\kappa \cdot \sigma \geq 1/\sqrt{2}} b_\delta(2(\kappa \cdot \sigma)^2 -1) \, \sqrt{1 - (\kappa \cdot \sigma)^2} \, |v - v_*|^{\gamma +1} \, M(v) \, m(v) \, \frac{1}{(\kappa \cdot \sigma)^{\gamma +3}} \, dv \, d\sigma \\
&\leq C  \int_{\R^3 \times \Sp^2} b_\delta(2(\kappa \cdot \sigma)^2 -1) \, \sqrt{1 - (\kappa \cdot \sigma)^2} \, |v - v_*|^{\gamma +1} \, M(v) \, m(v) \, dv \, d\sigma 
\end{aligned}
$$
where we used the fact that $\kappa \cdot \sigma \geq 1/ \sqrt{2}$ to bound from above $1/(\kappa \cdot \sigma)^{\gamma +3}$. Using the equalities
$$
\cos(2\theta) = 2(\kappa \cdot \sigma)^2 -1 \quad \text{and} \quad \sin \theta =\sqrt{1 - (\kappa \cdot \sigma)^2},
$$
we obtain
$$
\begin{aligned}
J &\leq C  \int_{\R^3 \times \Sp^2} b_\delta(\cos(2\theta)) \, \sin \theta \, |v-v_*|^{\gamma +1} \, M \, m \, dv \, d\sigma \\
&\leq C  \int_{\Sp^2} b_\delta(\cos(2\theta)) \, \sin \theta \, d\sigma \, \int_{\R^3} \langle v \rangle^{\gamma +1} \, M \, m \, dv \, \langle v_* \rangle^{\gamma+1} \\
& \leq C \, \kappa_\delta \, \langle v_* \rangle^{\gamma +1},
\end{aligned}
$$

Using this last estimate, we can conclude that 
\beqn \label{eq:I3}
I_3 \leq C \, \kappa_\delta \, \int_{\R^3} |h| \, \langle v \rangle^{\gamma+1} \, dv.
\eeqn

Gathering estimates (\ref{eq:L1delta}), (\ref{eq:L22delta}), (\ref{eq:I1}), (\ref{eq:I2}) and (\ref{eq:I3}), we can conclude that (\ref{eq:Ldeltadissip}) holds.
\end{proof}

We now want to deal with the part $\BB^c_{\delta, \varepsilon}-\nu_\delta$. To do that, we shall review a classical tool in the Boltzmann theory, a version of the Povzner lemma (see \cite{Wenn2,Bob,MW,BGP}). The version stated here is a consequence of the proof of Lemma 2.2 from \cite{MW}.

\begin{lem} \label{lem:Povzner}
For any $k>2$,
$$
\begin{aligned}
\forall \, v, \, v_* \in \R^3,  &\quad \int_{\Sp^2} \left[ \langle v'_* \rangle^k + \langle v'\rangle^k - \langle v_*\rangle^k - \langle v\rangle^k \right] \, b^c_\delta(\cos \theta) \, d\sigma \\
&\leq  C_k \,  \left( \langle v_*\rangle^{k-1} \langle v\rangle + \langle v\rangle^{k-1} \langle v_* \rangle\right) - C'_k \, |v|^k
\end{aligned}
$$
for some constants $C_k$, $C'_k>0$ depending on $k$. 
\end{lem}

\begin{proof}
If we adapt the proof of Lemma 2.2 from~\cite{MW} taking $\psi =\langle \cdot \rangle$, we obtain
$$
\begin{aligned}
&\quad \int_{\Sp^2} \left[  \langle v'_* \rangle^k + \langle v'\rangle^k - \langle v_*\rangle^k - \langle v\rangle^k \right] \, b^c_\delta(\cos \theta) \, d\sigma \\
&\leq C_k \,\left( \int_0^{\pi/2} b_\delta^c(\cos \theta) \sin^2(\theta) \, d\theta \right) \, 
 \left( \langle v_*\rangle^{k-1} \langle v\rangle + \langle v\rangle^{k-1} \langle v_* \rangle \right) -  C'_{k,\delta} \, |v|^k 
\end{aligned}
$$
with $C'_{k,\delta} \xrightarrow[\delta \rightarrow 0]{} +\infty$ and $C'_{k,\delta} \geq C'_k>0$ for any $\delta \in (0,1)$. 
We then conclude using~(\ref{eq:angularsing}) which implies that 
$$
\int_0^{\pi/2} b^c_\delta(\cos \theta) \sin^2(\theta) \, d\theta \approx \left(\frac{\pi}{2}\right)^{1-2s} - \delta^{1-2s} \leq C
$$
for any $\delta \in (0,1)$. 
\end{proof}

We can now prove the following estimate on $\BB^c_{\delta, \varepsilon}-\nu_\delta$.
\begin{lem} \label{lem:B1dissip}
For any $k>2$, for any $\varepsilon \in (0,1)$ and for $\delta \in (0,1)$ small enough, we have the following estimate
\beqn \label{eq:B1}
\forall \, h \in L^1(\langle v \rangle^\gamma m), \quad \int_{\R^3} \BB^c_{\delta, \varepsilon}  (h) \, \mathrm{sign} (h) \, m \, dv - \int_{\R^3} \nu_\delta \, |h| \, m \, dv
\leq \left(\Lambda_{k,\delta}(\varepsilon) - \lambda_k \right) \|h\|_{L^1(\langle v \rangle^\gamma m)} 
\eeqn
where $\lambda_k>0$ is a constant depending on $k$ and $\Lambda_{k,\delta}(\varepsilon)$ is a constant depending on $k$ and $\delta$ which tends to $0$ as $\varepsilon$ goes to $0$ when $k$ and $\delta$ are fixed. 
\end{lem}

\begin{proof}
We compute
$$
\|\BB^c_{\delta, \varepsilon} h \|_{L^1(m)} \leq \int_{\R^3 \times \R^3 \times \Sp^2} (1-\Theta_\varepsilon) \, B^c_\delta(v-v_*,\sigma) \left[ \mu'_* |h'| + \mu' |h'_*|+ \mu |h_*| \right] \, m \, d\sigma \, dv_* \, dv.
$$
We first bound from above the truncation function $(1-\Theta_\varepsilon)$:
$$
\begin{aligned}
\|\BB^c_{\delta, \varepsilon} h \|_{L^1(m)} 
&\leq \int_{\left\{|\cos \theta| \in [1-\varepsilon,1]\right\}} B^c_\delta(v-v_*,\sigma)\,  \mu_* |h| \, (m'+m'_*+m_*) \, d\sigma \, dv_* \, dv \\
&\quad +  \int_{\left\{|v-v_*| \leq \varepsilon\right\}} B^c_\delta(v-v_*,\sigma)\,  \mu_* |h| \, (m'+m'_*+m_*) \, d\sigma \, dv_* \, dv \\
&\quad + \int_{\left\{|v| \geq \varepsilon^{-1} \, \text{or} \,  |v-v_*| \geq \varepsilon^{-1}\right\}} B^c_\delta(v-v_*,\sigma) \left[ \mu'_* |h'| + \mu' |h'_*|+ \mu |h_*| \right] \, m \, d\sigma \, dv_* \, dv,
\end{aligned}
$$
where the pre-post collisional change of variables has been used in the two first terms. We obtain that $\|\BB^c_{\delta, \varepsilon} h \|_{L^1(m)}$ is bounded from above by
\beqn \label{eq:B11}
\begin{aligned}
& C_k \, \left( \int_{\left\{|\cos \theta| \in [1-\varepsilon,1] \right\}} \mathds{1}_{\theta \geq \delta} \, b(\cos \theta) \, d\sigma + K_\delta\,\varepsilon^\gamma \right) \int_{\R^3 \times \R^3} \mu_* \langle v_* \rangle^{\gamma +k} \, |h| \, \langle v \rangle^{\gamma+k} \, dv_* \, dv \\
&\quad + \int_{\R^3 \times \R^3 \times \Sp^2} \chi_{\varepsilon^{-1}} B^c_\delta(v-v_*,\sigma) \left[ \mu'_* |h'| + \mu' |h'_*|+ \mu |h_*| \right] \, m \, d\sigma \, dv_* \, dv \\
&=: J_1 + J_2
\end{aligned}
\eeqn
where $\chi_{\varepsilon^{-1}}$ is the characteristic function of the set
$$
\left\{ \sqrt{|v|^2 + |v_*|^2} \geq \varepsilon^{-1} \, \text{or}\, |v-v_*| \geq \varepsilon^{-1} \right\}.
$$

The first term of the right hand side of~(\ref{eq:B11}) is easily controlled as 
\beqn \label{eq:firstterms}
J_1 \leq C_k \,C_\delta \, \varepsilon^{\gamma} \, \|h\|_{L^1(\langle v \rangle^\gamma m)}.
\eeqn

As far as the second term in~(\ref{eq:B11}) is concerned, we write
$$
\begin{aligned}
J_2 &= \int_{\R^3 \times \R^3 \times \Sp^2} \chi_{\varepsilon^{-1}} B^c_\delta(v-v_*,\sigma) \left[ \mu'_* |h'| + \mu' |h'_*|+ \mu |h_*| \right] \, m \, d\sigma \, dv_* \, dv \\
&= \int_{\R^3 \times \R^3 \times \Sp^2} \chi_{\varepsilon^{-1}} B^c_\delta(v-v_*,\sigma) \left[ \mu'_* |h'| + \mu' |h'_*|- \mu_* |h| - \mu |h_*| \right] \, m \, d\sigma \, dv_* \, dv \\
&\quad + K_\delta \int_{\R^3 \times \R^3}  \chi_{\varepsilon^{-1}} \, \mu_* |h| \, |v-v_*|^\gamma \, m \, dv_* \, dv \\
&\quad + 2 \, K_\delta \int_{\R^3 \times \R^3}  \chi_{\varepsilon^{-1}} \, \mu |h_*| \, |v-v_*|^\gamma \, m \, dv_* \, dv \\
&=: T_1 + T_2 + T_3.
\end{aligned}
$$
We notice that the characteristic function $\chi_{\varepsilon^{-1}}$ is invariant under the usual pre-post collisional change of variables as it only depends on the kinetic energy and momentum. We hence bound the term $T_1$ thanks to Lemma~\ref{lem:Povzner}:
$$
\begin{aligned} 
T_1 &\leq \int_{\R^3 \times \R^3} \chi_{\varepsilon^{-1}} \, \mu_* |h| \, |v-v_*|^\gamma \, \int_{\Sp^2} \left(\langle v'_* \rangle ^k + \langle v' \rangle^k - \langle v_*\rangle^k - \langle v \rangle^k \right) b^c_\delta(\cos \theta) \, d\sigma \, dv_* \, dv \\
&\leq  C_k \, \int_{\R^3 \times \R^3}  \chi_{\varepsilon^{-1}} \, \mu_* |h| \, |v-v_*|^\gamma \, \left( \langle v \rangle^{k-1} \langle v_*\rangle +  \langle v\rangle \langle v_*\rangle^{k-1} \right) \, dv_* \, dv \\
&\quad - C'_k \, \int_{\R^3 \times \R^3} \chi_{\varepsilon^{-1}} \, \mu_* |h| \, |v-v_*|^\gamma \,  | v |^k \, dv_* \, dv \\
&\leq C_k \, \int_{\R^3 \times \R^3}  \chi_{\varepsilon^{-1}} \, \mu_* |h| \, |v-v_*|^\gamma \, \left(\langle v \rangle^{k-1}  \langle v_*\rangle +  \langle v\rangle \langle v_* \rangle^{k-1} \right) \, dv_* \, dv \\
&\quad +C'_k \, \int_{\R^3 \times \R^3} \chi_{\varepsilon^{-1}} \, \mu_* |h| \, |v-v_*|^\gamma \, dv_* \, dv \\
&\quad - C'_k \, 2^{1-k/2} \,\int_{\R^3 \times \R^3} \chi_{\varepsilon^{-1}} \, \mu_* |h| \, |v-v_*|^\gamma \, \langle v \rangle^k \, dv_* \, dv \\
&=: T_{11} + T_{12} + T_{13}. 
\end{aligned}
$$
We treat together the terms $T_{11}$, $T_{12}$ and $T_3$ using the following inequality:
$$
\chi_{\varepsilon^{-1}} (v,v_*) \leq \mathds{1}_{\left\{|v| \geq \varepsilon^{-1}/2\right\}} + \mathds{1}_{\left\{|v_*| \geq \varepsilon^{-1}/2\right\}} \leq 2 \, \varepsilon (|v| + |v_*|).
$$
We obtain:
\beqn \label{eq:T11T12T3}
\begin{aligned}
&\quad T_{11}+T_{12}+T_3 \\
&\leq \varepsilon \, C_k \, \int_{\R^3 \times \R^3} (|v|+|v_*|) \, \mu_* |h| \, |v-v_*|^\gamma \, \left(\langle v \rangle^{k-1} \langle v_* \rangle + \langle v \rangle \langle v_*\rangle^{k-1} \right) \, dv_* \, dv \\
&\quad +\varepsilon \, C'_k \, \int_{\R^3 \times \R^3} (|v|+|v_*|) \, \mu_* |h| \, |v-v_*|^\gamma \, dv_* \, dv \\
&\quad + \varepsilon \, K_\delta \int_{\R^3 \times \R^3}  (|v|+|v_*|) \, \mu |h_*| \, |v-v_*|^\gamma \, m \, dv_* \, dv \\
&\leq \varepsilon \,  C_k \,  C_\delta \, \|h\|_{L^1(\langle v \rangle^\gamma m)}.
\end{aligned}
\eeqn

Gathering~(\ref{eq:firstterms}) and~(\ref{eq:T11T12T3}), we conclude that 
\beqn \label{eq:positiveterm}
J_1 + T_{11} + T_{12} + T_3 \leq C_k \, C_\delta\, (\varepsilon + \varepsilon^\gamma) \, \|h\|_{L^1(\langle v \rangle^\gamma m)} 
=: \Lambda_{k,\delta} (\varepsilon) \, \|h\|_{L^1(\langle v \rangle^\gamma m)}.
\eeqn

We now put together the terms $T_{13}$, $T_2$ and the term coming from $\nu_\delta$, their sum is bounded from above by
$$
-K_\delta \int_{ \R^3\times \R^3} (1-\chi_{\varepsilon^{-1}}) \, \mu_* \, |v-v_*|^\gamma  |h| \, m \, dv_*  dv \, - \, C'_k \, 2^{1-k/2} \,\int_{ \R^3 \times \R^3} \chi_{\varepsilon^{-1}} \, \mu_* \,  |v-v_*|^\gamma  |h| \, m \, dv_* dv.
$$
Since $K_\delta \rightarrow \infty$ as $\delta \rightarrow 0$, we can take $\delta$ small enough so that $K_\delta \geq C'_k \, 2^{1-k/2}$, we obtain the following bound:
\beqn \label{eq:negativeterm}
-C'_k \, 2^{1-k/2} \int_{\R^3} (\mu \ast |\cdot|^\gamma) \, |h| \, m \, dv \leq - \lambda_k \, \|h\|_{L^1(\langle v \rangle^\gamma m)} .
\eeqn

Combining the bounds obtained in~(\ref{eq:positiveterm}) and (\ref{eq:negativeterm}), we can conclude that (\ref{eq:B1}) holds, which concludes the proof.
\end{proof}

We can now prove the dissipativity properties of $\BB_{\delta, \varepsilon} = \LL_\delta + \BB^c_{\delta, \varepsilon} - \nu_\delta$. 
\begin{lem} \label{lem:Bdissip}
Let us consider $a\in(-\lambda_k,0)$ where $\lambda_k$ is defined in~Lemma \ref{lem:B1dissip}. For $\delta>0$ and $\varepsilon>0$ small enough, $\BB_{\delta, \varepsilon} - a$ is dissipative in $L^1(m)$, namely
$$
\forall \, t \geq0, \quad \|S_{\BB_{\delta,\varepsilon}}(t)\|_{L^1(m) \rightarrow L^1(m)} \leq e^{at}.
$$
\end{lem}

\begin{proof}
Gathering results coming from lemmas~\ref{lem:Ldeltadissip} and~\ref{lem:B1dissip}, we obtain
$$
\int_{\R^3} \BB_{\delta, \epsilon} (h) \, \text{sign} (h) \, m \, dv \leq \int_{\R^3} \left( \varphi_k(\delta) + \Lambda_{k,\delta}(\varepsilon) -\lambda_k\right) |h| \, \langle v \rangle^\gamma \, m \, dv
$$
We first take $\delta$ small enough so that $\varphi_k(\delta) \leq (a+\lambda_k)/2$. We then chose $\varepsilon$ small enough so that $\Lambda_{k,\delta}(\varepsilon) \leq (a+\lambda_k)/2$. With this choice of $\delta$ and $\varepsilon$, we have the following inequality:
$$
\varphi_k(\delta) + \Lambda_{k,\delta}(\varepsilon) - \lambda_k \leq a.
$$
It implies that
$$
\int_{\R^3} \BB_{\delta, \epsilon} (h) \, \text{sign} (h) \, m \, dv \leq a \|h\|_{L^1(\langle v \rangle^\gamma m)},
$$
which concludes the proof. 
\end{proof}

\smallskip
\subsubsection{Regularization properties} \label{subsubsec:reg}
We first state a regularity estimate on the truncated operator $\AA_{\delta, \varepsilon}$ which comes from~\cite[Lemma~4.16]{GMM}. 
\begin{lem} \label{lem:Areg}
The operator $\AA_{\delta, \varepsilon}$ maps $L^1(\langle v \rangle)$ into $L^2$ functions with compact support. 
In particular, we can deduce that $\AA_{\delta, \varepsilon} \in \BBB\left(L^2\left(\mu^{-1/2}\right)\right)$ and $\AA_{\delta, \varepsilon} \in \BBB\left(L^1\left(m\right)\right)$. 
\end{lem}

We now study the regularization properties of $T(t): = \AA_{\delta, \varepsilon} \, S_{\BB_{\delta, \varepsilon}}(t)$. 
\begin{lem} \label{lem:Treg}
Consider $a \in (-\lambda_k,0)$. For a choice of $\delta$, $\varepsilon$ such that the conclusion of Lemma~\ref{lem:Bdissip} holds, there exists a constant $C>0$ such that 
$$
\|T(t) h\|_{L^2(\mu^{-1/2})} \leq C \, e^{at} \, \|h\|_{L^1(m)}.
$$
\end{lem}

\begin{proof}
We here use Lemma~\ref{lem:Areg}. We introduce a constant $R>0$ such that for any $h$ in $L^1(\langle v \rangle)$, $\text{supp} \, (\AA h) \subset B(0,R)$. We then compute
$$
\begin{aligned}
\| T(t) h \|_{L^2(\mu^{-1/2})} &\leq C  \left( \int_{B(0,R)} (T(t) h)^2 \,  dv \right)^{1/2} \leq C \, \|S_{\BB_{\delta, \varepsilon}}(t) h\|_{L^1(\langle v \rangle)} \\
&\leq C \, \|S_{\BB_{\delta, \varepsilon}}(t) h\|_{L^1(m)} \leq C \, e^{at} \, \|h\|_{L^1(m)},
\end{aligned}
$$
where the last inequality comes from Lemma~\ref{lem:Bdissip}. 
\end{proof}

\medskip

\subsection{Spectral gap in $L^1(\langle v \rangle^k)$} \label{subsec:L1}
\subsubsection{The abstract theorem}
Let us now present an enlargement of the functional space of a quantitative spectral mapping theorem (in the sense of semigroup decay estimate). The aim is to enlarge the space where the decay estimate on the semigroup holds. The version stated here comes from \cite[Theorem~2.13]{GMM}.

\begin{theo}
\label{th:abstract}
Let $E$, $\EE$ be two Banach spaces such that $E \subset \EE$ with dense and continuous embedding, and consider $L \in \CCC(E)$, $\LL \in \CCC(\EE)$ with  $\mathcal{L}_{|E} = L$ and $a \in \R$. 
We assume:
\begin{enumerate}
\item[{\bf (1)}]  $L$ generates a semigroup $S_L(t)$ and
$$ 
\Sigma(L) \cap \Delta_a = \left\{\xi \right\} \subset \Sigma_d(L) 
$$
for some $\xi \in \C$ and $L-a$ is dissipative on $\mathrm{R} \left(\mathrm{Id} - \Pi_{L,\xi}\right)$.
\item[{\bf (2)}]  There exist $\AA, \, \BB \in \CCC(\EE)$ such that $\LL = \AA + \BB$ (with corresponding restrictions $A$ and $B$ on $E$) and a constant $C_a >0$ so that 
\begin{enumerate}
\item[{\bf (i)}] $\BB - a$ is dissipative on $\EE$,
\item[{\bf (ii)}] $A \in \BBB(E)$ and $\AA \in \BBB(\EE)$, 
\item[{\bf (iii)}] $T(t) := \AA S_\BB(t)$ satisfies 
$$
\forall \,t \geq 0, \quad \left\|T(t)\right\|_{\BBB(\EE,E)} \leq C_a \, e^{at}.
$$
\end{enumerate}
\end{enumerate}
Then the following estimate on the semigroup holds:
$$
\forall \, a'>a, \, \forall \,  t \geq0, \quad \left\|S_\LL(t) -  S_L(t) \Pi_{\mathcal{L}, \xi} \right\|_{\BBB(\EE)} \leq C_{a'} e^{a't} 
$$
where $C_{a'}>0$ is an explicit constant depending on the constants from the assumptions.
\end{theo}

\subsubsection{Proof of Theorem~\ref{th:mainlinear}} 
The conclusion of Theorem~\ref{th:mainlinear} is a direct consequence of Theorem~\ref{th:abstract}. Indeed, denoting $E=L^2(\mu^{-1/2})$ and $\EE = L^1(m)$, assumption (1) is nothing but Proposition~\ref{prop:L2}, assumption (2)-(i) comes from Lemma~\ref{lem:Bdissip}, (2)-(ii) from Lemma~\ref{lem:Areg}  and (2)-(iii) from Lemma~\ref{lem:Treg}. We can conclude that estimate~(\ref{eq:sgdecay}) holds.

\bigskip

\section{The nonlinear equation} 
\label{sec:nonlin}
\setcounter{equation}{0}
\setcounter{theo}{0}


We first establish bilinear estimates on the collisional operator and we then prove our main result: Theorem~\ref{th:main}.

\medskip
\subsection{The bilinear estimates}
\begin{prop} \label{prop:bilinear}
Let $B$ satisfying (\ref{eq:product}), (\ref{eq:angularsing}) and (\ref{eq:Phi}). Then
$$
\|Q(h,h) \|_{L^1(m)} 
\leq C \, \Big( \|h\|_{L^1(\langle v \rangle^{\gamma } m)} \|h\|_{L^1(m)}   + \|h\|_{L^1(\langle v \rangle^{\gamma +1})} \|h\|_{W^{1,1}(\langle v \rangle^{\gamma +1} m)}   \Big)
$$
for some $C>0$. 
\end{prop}

\begin{proof}
We split $Q(h,h)$ into two parts and we use the pre-post collisional change of variables for the second one, we obtain
$$
\begin{aligned}
\|Q(h,h)\|_{L^1(m)} 
&= \int_{\R^3} \left| \int_{\R^3 \times \Sp^2} B(v-v_*, \sigma) \, \left( (h'_* -h_*)h + (h'-h)h'_* \right) \, d\sigma \, dv_* \, \right| m \, dv \\
&\leq \int_{\R^3} \left| \int_{\R^3 \times \Sp^2} B(v-v_*, \sigma) \,  (h'_* -h_*) \, d\sigma \, dv_* \right| |h| \, m \, dv \\
&\quad +  \int_{\R^3 \times \R^3 \times \Sp^2} B(v-v_*, \sigma) \, |h'-h| \, |h'_*| \, m \, d\sigma \, dv_* \, dv \\
&\leq  \int_{\R^3} \left| \int_{\R^3 \times \Sp^2} B(v-v_*, \sigma) \,  (h'_* -h_*) \, d\sigma \, dv_* \right| |h| \, m \, dv \\
&\quad +  \int_{\R^3 \times \R^3 \times \Sp^2} B(v-v_*, \sigma) \, |h'-h| \, |h_*| \, m' \, d\sigma \, dv_* \, dv \\
&=: T_1+T_2. 
\end{aligned}
$$

\smallskip
We first deal with $T_1$ using the cancellation lemma~\cite[Lemma~1]{ADVW}:
$$
T_1 = \int_{\R^3} \left| S \ast h \right| \, |h| \, m \, dv
$$
with
$$
\begin{aligned}
S(z) &= 2 \pi \, \int_0^{\pi/2} \sin\theta \, b(\cos \theta) \left( \frac{|z|^\gamma}{\cos^{\gamma+3} (\theta/2)} - |z|^\gamma \right) \, d\theta \\
&= 2 \pi \, |z|^\gamma \, \int_0^{\pi/2} \sin \theta  \, b(\cos \theta) \, \frac{1-\cos^{\gamma+3} ( \theta/2)}{\cos^{\gamma+3} (\theta/2)} \, d\theta \\
& \leq C \,  |z|^\gamma.
\end{aligned}
$$
We deduce that 
\beqn \label{eq:T1}
T_1 \leq C \, \|h\|_{L^1(\langle v \rangle^\gamma)} \|h\|_{L^1(\langle v \rangle^\gamma m)}.
\eeqn

\smallskip
We now treat the term $T_2$ which is splitted into two parts:
$$
\begin{aligned}
T_2 &= \int_{\R^3 \times \R^3 \times \Sp^2} B(v-v_*, \sigma) \, |h'm'-hm'| \, |h_*|  \, d\sigma \, dv_* \, dv \\
&\leq  \int_{\R^3 \times \R^3 \times \Sp^2} B(v-v_*, \sigma) \, |h'm'-hm| \, |h_*|  \, d\sigma \, dv_* \, dv \\
&\quad +  \int_{\R^3 \times \R^3 \times \Sp^2} B(v-v_*, \sigma) \, |m'-m| \, |h| \, |h_*|  \, d\sigma \, dv_* \, dv \\
&=: T_{21} + T_{22}.
\end{aligned}
$$

Concerning $T_{21}$, we have to estimate
$$
\int_{\R^3 \times \Sp^2} b(\cos \theta) \, |v-v_*|^\gamma \, |h'm'-hm| \, dv \, d\sigma =:J(v_*)=J.
$$
To do that, we use Taylor formula denoting $\overline{v}_u:=(1-u) v + u v'$ for any $u \in [0,1]$, which allows us to estimate $|h'm'-hm|$:
$$
\begin{aligned}
|h'm'-hm| &= \left| \int_0^1 \nabla (hm)(\overline{v}_u)  \cdot (v-v') \, du \right| \\
&\leq \int_0^1 |\nabla (hm) (\overline{v}_u)| \, |v-v_*| \, \sin (\theta/2) \, du.
\end{aligned}
$$ 
It implies the following inequality on $J$:
$$
J \leq C \int_{\R^3 \times \Sp^2 \times [0,1]} b(\cos \theta) \, \sin(\theta) \, |v-v_*|^{\gamma+1} \, \left|\nabla (hm) (\overline{v}_u)\right| \, du \, d\sigma \, dv.
$$
Moreover, if $v \neq v_*$, we have the following equality:
$$
|v-v_*|= \frac{1}{\left| \left(1-\frac{u}{2}\right) \, \kappa + \frac{u}{2} \, \sigma \right|} \, |\overline{v}_u-v_*|.
$$
Using the fact that $0 \leq \langle \kappa, \sigma \rangle \leq 1$, one can show that for any $u \in [0,1]$,   
$$
\left| \left(1-\frac{u}{2}\right) \, \kappa + \frac{u}{2} \, \sigma \right| \geq \frac{1}{\sqrt{2}}.
$$
We can thus deduce that for any $u \in [0,1]$, we have $|v-v_*| \leq C |\overline{v}_u-v_*|$ for some $C>0$, which implies 
$$
J \leq C \int_{\R^3 \times \Sp^2 \times [0,1]} b(\cos \theta) \, \sin(\theta) \, |\overline{v}_u-v_*|^{\gamma+1} \, \left|\nabla (hm) (\overline{v}_u)\right|\, du \, d\sigma \, dv.
$$
For $u$, $v_*$ and $\sigma$ fixed, we now perform the change of variables $v \rightarrow \overline{v}_u$. Its Jacobian determinant is 
$$
\left| \frac{d\overline{v}_u}{dv} \right| = \left(1- \frac{u}{2} \right)^2 \left( 1 - \frac{u}{2} + \frac{u}{2} \langle \kappa, \sigma \rangle \right) \geq \left( 1 - \frac{u}{2} \right)^3 \geq \frac{1}{8}
$$
since $\langle \kappa, \sigma \rangle \geq0$. Gathering all the previous estimates, we obtain
$$
J \leq C  \int_{\Sp^2} b(\cos \theta) \, \sin(\theta) \, d\sigma \, \int_{\R^3} |v-v_*|^{\gamma +1} \, |\nabla (hm)(v)| \, dv.
$$

We thus obtain :
\beqn \label{eq:T21}
T_{21} \leq C \, \|h\|_{L^1(\langle v \rangle^{\gamma+1})} \, \|h\|_{W^{1,1}(\langle v \rangle^{\gamma+1} m)}.
\eeqn

Let us finally deal with $T_{22}$. We here use the inequality~(\ref{eq:gradm}):
\beqn \label{eq:T22}
\begin{aligned} 
T_{22} &\leq C \, \int_{\R^3 \times \R^3 \times \Sp^2} B(v-v_*, \sigma) \, |h| \, |h_*| \, \left(\langle v \rangle^{k-1} + \langle v_* \rangle^{k-1} \right) \, |v'-v| \, d\sigma \, dv_* \, dv \\
&\leq C \, \int_{\Sp^2} b(\cos \theta) \, \sin(\theta) \, d\sigma \, \int_{\R^3 \times \R^3} |h| \, |h_*| \, \left(\langle v \rangle^{k-1} + \langle v_* \rangle^{k-1} \right) \, |v-v_*|^{\gamma+1} \, dv_* \, dv \\
&\leq C \, \|h\|_{L^1(\langle v \rangle^\gamma m)} \, \|h\|_{L^1(\langle v \rangle^{\gamma+1})}.
\end{aligned}
\eeqn

Inequalities (\ref{eq:T1}), (\ref{eq:T21}) and (\ref{eq:T22}) together yields the result. 
\end{proof}

\smallskip

We now recall a classical result from interpolation theory (see for example Lemma B.1 from~\cite{MM1}). 
\begin{lem} \label{lem:interp}
For any $s$, $s^*$, $q$, $q^* \in \Z$ with $s \geq s^*$, $q \geq q^*$ and any $\theta \in (0,1)$, there exists $C >0$ such that for any $h \in W^{s^{**},1}(\langle v \rangle^{q^{**}})$, we have
$$
\|h\|_{W^{s,1}(\langle v \rangle^q)} \leq C \, \|h\|^{1- \theta}_{W^{s^*,1}(\langle v \rangle^{q^*})} \,  \|h\|^\theta_{W^{s^{**},1}(\langle v \rangle^{q^{**}})} 
$$
with $s^{**}$, $q^{**} \in \Z$ such that $s = (1-\theta) s^* + \theta s^{**}$ and $q = (1-\theta) q^* + \theta q^{**}$.
\end{lem}

\smallskip
It allows us to prove the following corollary which is going to be useful in the proof of our main theorem.
\begin{cor} \label{cor:bilinear2}
Let $B$ satisfying (\ref{eq:product}), (\ref{eq:angularsing}) and (\ref{eq:Phi}). Then
$$
\begin{aligned}
\|Q(h,h) \|_{L^1(m)} &\leq C  \left( \|h\|_{L^1(m)}^{3/2} \, \|h\|_{L^1(\langle v \rangle^{2\gamma} m)}^{1/2} + \|h\|^{3/2}_{L^1(m)} \, \|h\|^{1/2}_{H^4(\langle v \rangle^{4 \gamma + k +6})}\right).
\end{aligned}
$$
\end{cor}
\begin{proof}
On the one hand, using Lemma~\ref{lem:interp}, we obtain:
$$
\|h\|_{L^1(\langle v \rangle^\gamma m)} \leq \|h\|^{1/2}_{L^1(\langle v \rangle^{2\gamma} m)} \, \|h\|_{L^1(m)}^{1/2}.
$$
On the other hand, again using twice Lemma~\ref{lem:interp}, we obtain
$$
\begin{aligned}
\|h\|_{L^1(\langle v \rangle^{\gamma+1})} \, \|h\|_{W^{1,1}(\langle v \rangle^{\gamma+1} m)} 
&\leq  \|h\|^2_{W^{1,1}(\langle v \rangle^{\gamma+k+1})} \\
&\leq C \, \|h\|_{L^1(m)} \, \|h\|_{W^{2,1}(\langle v \rangle^{2\gamma+k+2})} \\
&\leq C \, \|h\|^{3/2}_{L^1(m)} \, \|h\|^{1/2}_{W^{4,1}(\langle v \rangle^{4\gamma+k+4})}.
\end{aligned}
$$

To conclude we use that for any $q \in \N$, we can show using H\"{o}lder inequality that
$$
\|h\|_{L^1(\langle v \rangle^q)} \leq C \, \| h \|_{L^2(\langle v \rangle^{q+2})}.
$$
\end{proof}

\medskip
\subsection{Proof of Theorem \ref{th:main}} \label{subsec:main}

Let $f_0=\mu + h_0$ and consider the equation 
\beqn \label{eq:ht}
\partial_t  h_t = \LL h_t + Q(h_t,h_t), \quad h(t=0)=h_0.
\eeqn
Let us notice that for any $t\geq0$, we have $\Pi \, h_t=0$. Indeed, $f_0$ has same mass, momentum and energy as $\mu$, it implies that $\Pi \, h_0=0$ and these quantities are conserved by the equation.

\smallskip
We now state a nonlinear stability theorem which is the third key point (with Theorems~\ref{th:polydecay} and \ref{th:mainlinear}) in the proof of Theorem~\ref{th:main}.
\begin{theo} \label{th:nonlinearstability}
Consider a solution $h_t$ to~(\ref{eq:ht}) such that 
$$
\forall \, t \geq 0, \quad \|h_t\|_{H^4(\langle v \rangle^{4 \gamma + k + 6})} \leq K
$$
for some $K>0$. 
There exists $\eta>0$ such that if moreover
$$
\forall \, t \geq 0, \quad \|h_t\|_{L^1(\langle v \rangle^{2\gamma}m)} \leq \eta
$$
then there exists $C>0$ (depending on $K$ and $\eta$) such that 
$$
\forall \, t \geq0, \quad \|h_t\|_{L^1(m)} \leq C \, e^{-\lambda t} \, \|h_0\|_{L^1(m)}
$$
for any positive $\lambda < \min (\lambda_0, \lambda_k)$ (see Theorem~\ref{th:mainlinear}).
\end{theo}

\begin{proof}
We use Duhamel's formula for the solution of~(\ref{eq:ht}):
$$
h_t = S_\LL(t) \,h_0 + \int_0^t S_\LL(t-s) \, Q(h_s, h_s) \, ds.
$$
We now estimate $\|h_t\|_{L^1(m)}$ thanks to Theorem~\ref{th:mainlinear} and Corollary~\ref{cor:bilinear2}:
$$
\begin{aligned}
\|h_t\|_{L^1(m)} 
&\le e^{-\lambda t } \|h_0\|_{L^1(m)} \\
&\quad+  C \int_0^t e^{-\lambda (t-s)} \Big(\|h_s\|^{1/4}_{L^1( m)} \,  \|h_s\|^{1/2}_{H^4(\langle v \rangle^{4\gamma+k+6})}  +\|h_s\|_{L^1(\langle v \rangle^{2\gamma} m)}^{3/4}   \Big) \|h_s\|^{5/4}_{L^1(m)} \, ds \\
&\le e^{-\lambda t } \, \|h_0\|_{L^1(m)} + C \int_0^t e^{-\lambda (t-s)} \left( K^{1/2} \eta^{1/4} + \eta^{3/4} \right) \|h_s\|^{5/4}_{L^1(m)} \, ds.
\end{aligned}
$$
We denote $\eta':=C \left( K^{1/2} \eta^{1/4} + \eta^{3/4} \right)$. We end up with a similar differential inequality as in \linebreak \cite[~Lemma~4.5]{Mou2}. We can then conclude in the same way that 
$$
\forall \, t \geq 0, \quad \|h_t\|_{L^1(m)} \le C' e^{-\lambda t} \|h_0\|_{L^1(m)},
$$
for some $C'>0$.
\end{proof}

\smallskip
To conclude the proof of Theorem~\ref{th:main}, we consider $\eta>0$ defined in Theorem~\ref{th:nonlinearstability}. Using Theorem~\ref{th:polydecay}, we can choose $t_1>0$ such that 
$$
\forall \, t \geq t_1, \quad \|h_t\|_{L^1(m)} = \|f_t - \mu\|_{L^1(m)} \le \eta.
$$
Thanks to the properties of a smooth solution, we also have 
$$ 
\forall \, t \geq t_1, \quad \|h_t\|_{H^4(\langle v \rangle^{4 \gamma + k + 6})} \leq  \|f_t\|_{H^4(\langle v \rangle^{4 \gamma + k + 6})} +  \|\mu\|_{H^4(\langle v \rangle^{4 \gamma + k + 6})} \leq K
$$ 
for some $K>0$. We can hence apply Theorem~\ref{th:nonlinearstability} to $h_t$ starting from $t_1$. We finally obtain
$$
\forall \, t \geq t_1, \quad \|f_t - \mu\|_{L^1(m)} \leq C' e^{-\lambda t} \|h_{t_1}\|_{L^1(m)} \leq C'' e^{-\lambda t},
$$
for some $C''>0$. The conclusion of Theorem~\ref{th:main} is hence established. 

\bigskip
\bibliographystyle{acm}

\end{document}